\documentclass[12pt,twoside]{amsart}
\usepackage{amssymb,amscd,amsxtra,calc,mathrsfs}
\usepackage{amsmath}
\usepackage{xypic}

\title[K-unstable Fano manifolds]{Examples of K-unstable Fano manifolds 
with the Picard number one}
\author{Kento Fujita} 
\date{\today}
\address{Department of Mathematics, Faculty of Science, 
Kyoto University, Kyoto 606-8502, Japan}
\email{fujita@math.kyoto-u.ac.jp}

\newcommand{\pr}{\mathbb{P}}

\newcommand{\Z}{\mathbb{Z}}
\newcommand{\Q}{\mathbb{Q}}
\newcommand{\R}{\mathbb{R}}
\newcommand{\C}{\mathbb{C}}
\newcommand{\F}{\mathbb{F}}

\newcommand{\A}{\mathbb{A}}
\newcommand{\G}{\mathbb{G}}

\newcommand{\ND}{\operatorname{N}^1}

\newcommand{\Eff}{\operatorname{Eff}}

\newcommand{\DIV}{\operatorname{div}}

\newcommand{\id}{\operatorname{id}}

\newcommand{\DF}{\operatorname{DF}}

\newcommand{\vol}{\operatorname{vol}}
\newcommand{\sI}{\mathcal{I}}

\newcommand{\sO}{\mathcal{O}}
\newcommand{\sN}{\mathcal{N}}

\newcommand{\sL}{\mathcal{L}}

\newcommand{\sB}{\mathcal{B}}

\newtheorem{thm}{Theorem}[section]
\newtheorem{lemma}[thm]{Lemma}

\theoremstyle{definition}
\newtheorem{definition}[thm]{Definition}
\newtheorem{remark}[thm]{Remark}

\newtheorem{assumption}[thm]{Assumption}

\newtheorem*{ack}{Acknowledgments}

\begin{document}

\maketitle 

\begin{abstract}
We show that the pair $(X, -K_X)$ is K-unstable 
for a del Pezzo manifold $X$ of degree five with dimension four or five. 
This disprove a conjecture of Odaka and Okada. 
\end{abstract}

\setcounter{tocdepth}{1}
\tableofcontents

\section{Introduction}\label{intro_section}

Let $X$ be a \emph{Fano manifold}, i.e., $X$ is a smooth projective variety 
over the complex number filed $\C$ with $-K_X$ 
ample. We have known that the 
$X$ admits K\"ahler-Einstein metrics if and only if 
the pair $(X, -K_X)$ is K-polystable
(\cite{tian1, don05, CT, stoppa, mab1, mab2, Berman, CDS1, CDS2, CDS3, tian2}). 
K-stabe implies K-polystable, and K-polystable implies K-semistable. 
(See Section \ref{prelim_section} for the definitions.) 
In this article, we consider K-semistability of the pair $(X, -K_X)$ such 
that the Picard number $\rho(X)$ of $X$ is equal to one. 
By \cite[Section 7]{tian1}, a small deformation of the Mukai-Umemura 
threefold does not admit K\"ahler-Einstein metrics. 
However, by \cite[Theorem 1.1 (i)]{LWX}, the $X$ satisfies that the pair $(X, -K_X)$ 
is K-semistable since the Mukai-Umemura threefold admits K\"ahler-Einstein metrics 
\cite[Corollary 1]{donMU}. As in \cite[Conjecture 5.1]{OO} 
(see also \cite[Question 1.12]{CFST}), no example has been 
known so far that a Fano manifold $X$ of $\rho(X)=1$ with $(X, -K_X)$ K-unstable, 
that is, non-K-semistable. 

In this article, we find counterexamples of \cite[Conjecture 5.1]{OO}, that is, 
we find Fano manifolds $X$ of $\rho(X)=1$ with $(X, -K_X)$ K-unstable.

\begin{thm}\label{mainthm}
Let $X$ be a del Pezzo manifold of degree five in the sense of \cite[I, \S 8]{fujita} 
$($see Section \ref{dP_section}$)$. 
If $\dim X=4$ or $5$, then the pair $(X, -K_X)$ is K-unstable, that is, not K-semistable. 
\end{thm}

\begin{remark}[{\cite{isk, F}}]\label{mainrmk}
Fix $n\geq 2$. Let $X$ be a del Pezzo manifold of degree five and $\dim X=n$. 
\begin{enumerate}
\renewcommand{\theenumi}{\arabic{enumi}}
\renewcommand{\labelenumi}{(\theenumi)}
\item\label{mainrmk1}
All the $X$ of fixed $\dim X=n$ 
are isomorphic to each other. Moreover, we have $n\leq 6$. 
\item\label{mainrmk2}
Any such $X$ is isomorphic to a linear section of $\G(5, 2)$ 
embedded by the Pl\"ucker coordinate. Here $\G(5, 2)$ denotes the Grassmann 
variety parameterizing two-dimensional linear subspaces of $\C^5$. 
\end{enumerate}
\end{remark}

\begin{remark}\label{dPnrmk}
Let $X$ be a del Pezzo manifold of degree five and $\dim X=n\geq 2$. 
\begin{enumerate}
\renewcommand{\theenumi}{\arabic{enumi}}
\renewcommand{\labelenumi}{(\theenumi)}
\item\label{dPrmk1}
If $n=6$, then $X\simeq\G(5, 2)$. Thus $X$ is a rational homogeneous manifold. 
Hence $X$ admits K\"ahler-Einstein metrics. In particular, the pair $(X, -K_X)$ is 
K-semistable. 
\item\label{dPrmk2}
If $n=3$, then $X$ admits K\"ahler-Einstein metrics by \cite[Theorem 1.17]{CS}. 
In particular, the pair $(X, -K_X)$ is K-semistable. 
\item\label{dPrmk3}
If $n=2$, then $X$ admits K\"ahler-Einstein metrics by \cite[Theorem 1.1]{tian}. 
In particular, the pair $(X, -K_X)$ is K-semistable. 
(We note that $\rho(X)=5$ if $n=2$.)
\end{enumerate}
\end{remark}

Thus, together with Theorem \ref{mainthm}, we have completely checked 
K-semistability of the pairs $(X, -K_X)$ for del Pezzo manifolds $X$ of degree five. 

We outline the proof of Theorem \ref{mainthm}. We know that, by the results of 
Takao Fujita \cite{F}, a del Pezzo manifold $X$ of degree five with $\dim X\geq 4$ 
has a special plane (see Theorem \ref{dP_thm}). By the same argument in \cite{fjt}, 
we get a semi test configuration from the plane (see Section \ref{btc_section}). 
We calculate the Donaldson-Futaki invariant of the 
semi test configuration in Section \ref{proof_section}.

\begin{ack}
The author is partially supported by a JSPS Fellowship for Young Scientists. 
\end{ack}

In this article, a \emph{variety} means a reduced, irreducible, separated and of 
finite type scheme over the complex number field $\C$. 
For a normal projective variety $X$, the Picard number of $X$ is denoted by $\rho(X)$; 
the closure of the cone spanned by the classes of effective Cartier divisors on $X$ 
in $\ND(X)$ is denoted by 
$\overline{\Eff}(X)$. Moreover, 
For a Weil divisor $D$ on $X$, the 
\emph{divisorial sheaf} on $X$ is denoted by $\sO_X(D)$. More precisely, 
the section of $\sO_X(D)$ on any open subscheme $U\subset X$ is defined by 
the following:
\[
\{f\in k(X)\,|\,\DIV(f)|_U+D|_U\geq 0\},
\]
where $k(X)$ is the function field of $X$. 

Let $\F_1$ be the blowup of $\pr^2$ along a point. Let $e\subset \F_1$ be the 
$(-1)$-curve and let $l\subset \F_1$ be the strict transform of a line on $\pr^2$ 
passing through the center of the blowup. 
For varieties $X_1$ and $X_2$, let $p_i\colon X_1\times X_2\to X_i$ $(i=1$, $2)$ be 
the natural projection morphism.

\section{Preliminaries}\label{prelim_section}

\subsection{K-stability}\label{K_section}

We recall the definition of K-stability. 

\begin{definition}[{see \cite{tian1, don, RT, odk, odk15}}]\label{K_dfn}
Let $X$ be a Fano manifold of dimension $n$. 
\begin{enumerate}
\renewcommand{\theenumi}{\arabic{enumi}}
\renewcommand{\labelenumi}{(\theenumi)}
\item\label{K_dfn1}
If a coherent ideal sheaf $\sI\subset\sO_{X\times\A_t^1}$ satisfies that 
$\sI$ is of the form 
\[
\sI=I_M+I_{M-1}t^1+\cdots+I_1t^{M-1}+(t^M)\subset\sO_{X\times\A_t^1}
\]
($I_M\subset\cdots\subset I_1\subset\sO_X$ is a sequence of coherent ideal sheaves 
of $X$), then we call the $\sI$ a \emph{flag ideal}. 
\item\label{K_dfn2}
Let $r\in\Z_{>0}$ and let $\sI\subset\sO_{X\times\A^1}$ be a flag ideal. 
A \emph{semi test configuration} $(\sB, \sL)/\A^1$ \emph{of} 
$(X, -rK_X)$ \emph{obtained by} $\sI$ is given from: 
\begin{itemize}
\item
the morphism $\Pi\colon\sB\to X\times\A^1$ which is the blowup along $\sI$, 
and $E_\sB\subset\sB$ 
is the Cartier divisor defined by $\sO_\sB(-E_\sB)=\sI\cdot\sO_\sB$,
\item
$\sL$ is the line bundle on $\sB$ defined by the equation 
$\sL:=\Pi^*p_1^*\sO_X(-rK_X)\otimes\sO_\sB(-E_\sB)$, 
\end{itemize}
such that we require that 
\begin{itemize}
\item
$\sI$ is not of the form $(t^M)$, and
\item
$\sL$ is semiample  over $\A^1$.
\end{itemize}
\item\label{K_dfn3}
Let $(\sB, \sL)/\A^1$ be the semi test configuration of $(X, -rK_X)$
obtained by $\sI$. 
For $k\in\Z_{>0}$, set 
\[
w(k):=-\dim\left(\frac{H^0\left(X\times\A^1, p_1^*\sO_X(-krK_X)\right)}
{H^0\left(X\times\A^1, p_1^*\sO_X(-krK_X)\cdot\sI^k\right)}\right).
\]
It is known that $w(k)$ is a polynomial function 
of degree at most $n+1$ for $k\gg 0$. Let $w_{n+1}$(resp.\ $w_n$) be the $(n+1)$-th 
(resp.\ $n$-th) coefficient of $w(k)$. We set the 
\emph{Donaldson-Futaki invariant} $\DF(\sB, \sL)$ of $(\sB, \sL)/\A^1$ with 
\begin{eqnarray*}
\DF(\sB, \sL) & := & \frac{\left((-rK_X)^{\cdot n-1}\cdot(-K_X)\right)}
{2\cdot(n-1)!}w_{n+1}-
\frac{((-rK_X)^{\cdot n})}{n!}w_n\\
 & = & \frac{((-rK_X)^{\cdot n})}{n!}\left(\frac{n}{2r}w_{n+1}-w_n\right).
\end{eqnarray*}
\item\label{K_dfn4}
The pair $(X, -K_X)$ is said to be \emph{K-stable} (resp.\ \emph{K-semistable}) 
if the inequality 
$\DF(\sB, \sL)>0$ (resp.\ $\geq 0$) holds for any $r\in\Z_{>0}$, for any flag ideal 
$\sI$, and for any semi test configuration $(\sB, \sL)/\A^1$ of $(X, -rK_X)$ 
obtained by $\sI$. 
The pair $(X, -K_X)$ is said to be \emph{K-unstable} if the pair 
$(X, -K_X)$ is not K-semistable. 
\end{enumerate}
\end{definition}

\subsection{Del Pezzo manifolds of degree five}\label{dP_section}

We recall the definition and properties of del Pezzo manifolds. 

\begin{definition}[{\cite[I, \S 8]{fujita}}]\label{dP_dfn}
Let $(X, L)$ be an $n$-dimensional polarized manifold, that is, $X$ is an 
$n$-dimensional smooth projective variety and $L$ is an ample divisor on $X$. 
Assume that $n\geq 2$ and 
let $d$ be a positive integer. 
The pair $(X, L)$ is said to be a \emph{del Pezzo manifold} of degree $d$ if 
$-K_X\sim (n-1)L$ and $(L^{\cdot n})=d$. We often omit the polarization $L$ 
since the divisor $L$ is unique up to linear equivalence. 
By \cite[(6.2.3)]{fujita}, if $d\geq 3$ then the complete linear system $|L|$ defines 
an embedding $X\hookrightarrow\pr^{n+d-2}$. We sometimes write 
$X\subset\pr^{n+d-2}$ in place of $(X, L)$ if $d\geq 3$. 
\end{definition}

The following structure results are essential in this article. 

\begin{thm}[{\cite{F}}]\label{dP_thm}
Let $X\subset\pr^{n+3}$ be a del Pezzo manifold of dimension $n$ and degree five. 
Let $L$ be a divisor on $X$ with $-K_X\sim (n-1)L$. 
\begin{enumerate}
\renewcommand{\theenumi}{\arabic{enumi}}
\renewcommand{\labelenumi}{(\theenumi)}
\item\label{dP_thm1}
Assume that $n=4$. Then there exists a plane $S\subset X$ such that 
$c_2(\sN_{S/X})=2$. Let $\sigma\colon\tilde{X}\to X$ be the blowup along $S$ and 
let $E_S$ be the exceptional divisor. 
Then the complete linear system $|\sigma^*L-E_S|$ gives a birational surjection 
$\pi\colon\tilde{X}\to\pr^4$. Moreover, the push-forward $H:=\pi_*E_S$ is a hyperplane 
in $\pr^4$ and the morphism $\pi$ is the blowup along 
a twisted cubic curve $C$ on $H$. 
\item\label{dP_thm2}
Assume that $n=5$. Then there exists a plane $S\subset X$ such that 
$c_2(\sN_{S/X})=2$. Let $\sigma\colon\tilde{X}\to X$ be the blowup along $S$ and 
let $E_S$ be the exceptional divisor. 
Then the complete linear system $|\sigma^*L-E_S|$ gives a birational surjection 
$\pi\colon\tilde{X}\to\pr^5$. Moreover, the push-forward $H:=\pi_*E_S$ is a hyperplane 
in $\pr^5$, the morphism $\pi$ is the blowup along 
a smooth surface $C$ on $H$ with $C\simeq\F_1$ and the embedding 
$C\subset H$ is obtained by the complete linear system $|e+2l|$. 
\item\label{dP_thm3}
Assume that $n=6$. Then there exists a plane $S\subset X$ such that 
$c_2(\sN_{S/X})=3$. Let $\sigma\colon\tilde{X}\to X$ be the blowup along $S$ and 
let $E_S$ be the exceptional divisor. 
Then the complete linear system $|\sigma^*L-E_S|$ gives a birational surjection 
$\pi\colon\tilde{X}\to\pr^6$. Moreover, the push-forward $H:=\pi_*E_S$ is a hyperplane 
in $\pr^6$, the morphism $\pi$ is the blowup along 
a smooth threefold $C$ on $H$ with $C\simeq\pr^1\times\pr^2$ and the embedding 
$C\subset H$ is obtained by the complete linear system $|\sO_{\pr^1\times\pr^2}(1,1)|$. 
\end{enumerate}
\end{thm}

\section{Basic semi test configurations via submanifolds}\label{btc_section}

In this section, we construct specific semi test configurations from submanifolds under 
some extra conditions. 
The strategy for the construction 
is essentially in the same way as that in \cite[\S 3]{fjt}. 
In this section, we fix the following condition: 

\begin{assumption}\label{fg_assump}
Let $X$ be a Fano manifold of dimension $n$, let $S\subset X$ be a smooth subvariety 
of codimension $d$ corresponds to an ideal sheaf $I_S\subset \sO_X$, let 
$\sigma\colon\tilde{X}\to X$ be the blowup along $S$ and let $E_S$ be the exceptional 
divisor. We assume that the $\Z^{\oplus 2}_{\geq 0}$-graded $\C$-algebra
\[
\bigoplus_{k,j\geq 0}H^0\left(\tilde{X}, \sO_{\tilde{X}}\left(k\sigma^*(-K_X)
-jE_S\right)\right)
\]
is finitely generated. 
\end{assumption}

\begin{remark}\label{BCHM_rmk}
If $\tilde{X}$ is a Fano manifold, then the above $\C$-algebra is finitely generated 
by \cite[Corollary 1.3.2]{BCHM}. 
\end{remark}

\subsection{Geography of models}\label{geog_section}

We study the theory of ``geography of models" introduced in \cite{shokurov}. 
We use the notations in \cite{KKL} (see also \cite[\S 2.2]{fjt}).

\begin{definition}\label{et_dfn}
Under Assumption \ref{fg_assump}, set
\[
\tau(S):=\max\{\tau\in\R_{>0}\,\,|\,\,\sigma^*(-K_X)-\tau E_S\in
\overline{\Eff}(\tilde{X})\}.
\]
\end{definition}

\begin{thm}[{\cite[Theorem 4.2]{KKL}}]\label{KKL_thm}
Under Assumption \ref{fg_assump}, 
there exist 
\begin{itemize}
\item
a strictly increasing sequence of rational numbers
\[
0=\tau_0<\tau_1<\cdots<\tau_m=\tau(S)
\] 
$($in particular, $\tau(S)\in\Q_{>0}$ holds$)$, 
\item
normal projective varieties
$X_1,\dots,X_m$ with $X_1=\tilde{X}$, and 
\item
mutually distinct birational contraction maps
$\phi_i\colon\tilde{X}\dashrightarrow X_i$  
$(1\leq i\leq m)$ with $\phi_1=\id_{\tilde{X}}$
\end{itemize}
such that the following conditions are satisfied: 
\begin{itemize}
\item
for any $x\in[\tau_{i-1}, \tau_i]$, the birational contraction map $\phi_i$ is 
a semiample model 
$($see \cite[Definition 2.3]{KKL}$)$ of $\sigma^*(-K_X)-xE_S$, and 
\item
if $x\in(\tau_{i-1}, \tau_i)$, then the birational contraction map $\phi_i$ is 
the ample model 
$($see \cite[Definition 2.3]{KKL}$)$ of $\sigma^*(-K_X)-xE_S$. 
\end{itemize}
\end{thm}

\begin{definition}\label{model_dfn}
The sequence $\{(\tau_i, X_i)\}_{1\leq i\leq m}$ obtained from Theorem \ref{KKL_thm} 
is called the \emph{ample model sequence of} $(X, -K_X; I_S)$. We set 
$E_i:=(\phi_i)_*E_S$ for $1\leq i\leq m$. 
\end{definition}

\subsection{Construction of basic semi test configurations}

Under Assumption \ref{fg_assump}, the graded $\C$-algebra in 
Assumption \ref{fg_assump} is equal to 
\[
\bigoplus_{\substack{k\geq 0,\\ 0\leq j\leq k\tau(S)}}
H^0\left(X, \sO_X(-kK_X)\cdot I_S^j\right).
\]
(see \cite[Lemma 4.3.16]{L}). Indeed, $H^0(X, \sO_X(-kK_X)\cdot I_S^j)=0$ if 
$k\geq 0$ and $j>k\tau(S)$ by the definition of $\tau(S)$. 
Fix a positive integer $r\in\Z_{>0}$ such that the following hold: 
\begin{itemize}
\item
$r\tau(S)\in\Z_{>0}$, and
\item
the graded $\C$-algebra
\[
\bigoplus_{\substack{k\geq 0,\\ 0\leq j\leq kr\tau(S)}}
H^0\left(X, \sO_X(-krK_X)\cdot I_S^j\right)
\]
is generated by 
\[
\bigoplus_{0\leq j\leq r\tau(S)}H^0\left(X, \sO_X(-rK_X)\cdot I_S^j\right)
\]
as a $\C$-algebra. 
\end{itemize}
If $r\in\Z_{>0}$ is sufficiently divisible, then the $r$ satisfies the above conditions. 
We remark that the $\C$-algebra
\[
\bigoplus_{k\geq 0}H^0\left(X, \sO_X(-krK_X)\right)
\]
is generated by $H^0(X, \sO_X(-rK_X))$ by the choice of $r$. This implies that 
the divisor $-rK_X$ is very ample.

From now on, we construct a semi test configuration of $(X, -rK_X)$ from $S$ and $r$. 
For any $j\in\Z_{\geq 0}$, we set the coherent ideal sheaf $I_j\subset\sO_X$ 
defined by the image of the following: 
\begin{eqnarray*}
H^0\left(X, \sO_X(-rK_X)\cdot I_S^j\right)\otimes_\C\sO_X(rK_X)\to\sO_X.
\end{eqnarray*}
The ideal sheaf $I_j\subset\sO_X$ is nothing but the base ideal of the 
sub linear system of $|-rK_X|$ corresponds to the inclusion
\[
H^0(X, \sO_X(-rK_X)\cdot I_S^j)\subset H^0(X, \sO_X(-rK_X)).
\]
For $k\in\Z_{>0}$ and $j\in\Z_{\geq 0}$, we define 
\[
J_{(k, j)}:=\sum_{\substack{j_1+\cdots+j_k=j,\\ j_1,\dots,j_k\geq 0}}I_{j_1}\cdots I_{j_k}
\subset\sO_X.
\]

\begin{lemma}[{cf.\ \cite[Lemma 3.3]{fjt}}]\label{gen_lem}
The above ideal sheaf $J_{(k, j)}\subset\sO_X$ is equal to 
the image of the following:
\begin{eqnarray*}
H^0\left(X, \sO_X(-krK_X)\cdot I_S^j\right)\otimes_\C\sO_X(krK_X)\to\sO_X.
\end{eqnarray*}
In particular, 
\[
H^0\left(X, \sO_X(-krK_X)\cdot I_S^j\right)
=H^0\left(X, \sO_X(-krK_X)\cdot J_{(k, j)}\right)
\]
holds as subspaces of $H^0(X, \sO_X(-krK_X))$.
\end{lemma}

\begin{proof}
We set 
\[
V_{k,j}:=H^0(X, \sO_X(-krK_X)\cdot I_S^j)
\]
for simplicity. Note that the homomorphism 
\[
\bigoplus_{\substack{j_1+\cdots+j_k=j,\\ j_1,\dots,j_k\geq 0}}
V_{1, j_1}\otimes_\C\cdots\otimes_\C V_{1, j_k}\to V_{k, j}
\]
is surjective by the choice of $r\in\Z_{>0}$. For any $1\leq i\leq k$, 
recall that the ideal sheaf $I_{j_i}$ is equal to 
the image of the homomorphism 
\[
V_{1, j_i}\otimes_\C\sO_X(rK_X)\to \sO_X.
\]
Thus we get the assertion. 
\end{proof}

We set the flag ideal $\sI$ of the form
\[
\sI:=I_{r\tau(S)}+I_{r\tau(S)-1}t^1+\cdots+ I_1t^{r\tau(S)-1}+(t^{r\tau(S)})
\subset\sO_{X\times\A^1_t}.
\]
(Since 
\[
0=I_{r\tau(S)+1}\subset I_{r\tau(S)}\subset\cdots\subset I_1\subset I_0=\sO_X,
\]
the $\sI$ is actually an ideal sheaf.)
For any $k\in\Z_{>0}$, we have the equality
\[
\sI^k=J_{(k, kr\tau(S))}+J_{(k, kr\tau(S)-1)}t^1+\cdots+J_{(k, 1)}t^{kr\tau(S)-1}
+(t^{kr\tau(S)})
\]
by the definition of $J_{(k,j)}$.
Let $\Pi\colon\sB\to X\times\A^1$ be the blowup along $\sI$ and 
$E_\sB\subset\sB$ be the Cartier divisor given by the equation 
$\sO_\sB(-E_\sB)=\sI\cdot\sO_\sB$. Set 
$\sL:=\Pi^*p_1^*\sO_X(-rK_X)\otimes\sO_\sB(-E_\sB)$. 

\begin{lemma}[{cf.\ \cite[Lemma 3.4]{fjt}}]\label{stc_lem}
$(\sB, \sL)/\A^1$ is a semi test configuration of the pair $(X, -rK_X)$. 
\end{lemma}

\begin{proof}
Set $\alpha:=p_2\circ\Pi\colon\sB\to\A^1$. 
It is enough to show that $\sL$ is $\alpha$-semiample. 
For any $k\in\Z_{>0}$ and $j\in\Z_{\geq 0}$, the homomorphism 
\[
H^0(X, \sO_X(-krK_X)\cdot J_{(k, j)})\otimes_\C\sO_X\to\sO_X(-krK_X)\cdot J_{(k, j)}
\]
is surjective by Lemma \ref{gen_lem}. This implies that the homomorphism 
\[
H^0\left(X\times\A^1, p_1^*\sO_X(-krK_X)\cdot\sI^k\right)\otimes_{\C[t]}
\sO_{X\times\A^1}\to p_1^*\sO_X(-krK_X)\cdot \sI^k
\]
is surjective for any $k\in\Z_{>0}$. 
From \cite[Lemma 5.4.24]{L}, we have 
\begin{eqnarray*}
\alpha^*\alpha_*\sL^{\otimes k}&\simeq& 
\alpha^*(p_2)_*(p_1^*\sO_X(-krK_X)\cdot\sI^k)\\
&=&\Pi^*\left(H^0\left(X\times\A^1, p_1^*\sO_X(-krK_X)\cdot\sI^k\right)
\otimes_{\C[t]}\sO_{X\times\A^1}\right)\\
&\twoheadrightarrow&\Pi^*\left(p_1^*\sO_X(-krK_X)\cdot\sI^k\right)\\
&\twoheadrightarrow&
\Pi^*p_1^*\sO_X(-krK_X)\otimes\sO_\sB(-kE_\sB)=\sL^{\otimes k} 
\end{eqnarray*}
for $k\gg 0$. 
This implies that $\sL$ is $\alpha$-semiample. 
\end{proof}

\begin{definition}\label{divst_dfn}
We say the above semi test configuration $(\sB, \sL)/\A^1$ the 
\emph{basic semi test configuration of} $(X, -rK_X)$ \emph{via} $S$. 
\end{definition}

\subsection{Calculating the Donaldson-Futaki invariants}\label{calc_section}

Let $(\sB, \sL)/\A^1$ be the semi test configuration of $(X, -rK_X)$ 
via $S$. Then the value $w(k)$ in Definition \ref{K_dfn} \eqref{K_dfn3} is equal to
\begin{eqnarray*}
-kr\tau(S)h^0\left(X, \sO_X(-krK_X)\right)+\sum_{j=1}^{kr\tau(S)}
h^0\left(X, \sO_X(-krK_X)\cdot J_{(k, j)}\right).
\end{eqnarray*}
For $k\gg 0$ sufficiently divisible, by \cite[Remark 2.4 (i)]{KKL} and 
\cite[Proposition 4.1]{fjt}, we have 
\begin{eqnarray*}
&&\sum_{j=1}^{kr\tau(S)}h^0\left(X, \sO_X(-krK_X)\cdot J_{(k, j)}\right)\\
&=&\sum_{j=1}^{kr\tau(S)}h^0\left(\tilde{X}, \sO_{\tilde{X}}
(kr\sigma^*(-K_X)-jE_S)\right)\\
&=&\sum_{i=1}^m\sum_{j=kr\tau_{i-1}+1}^{kr\tau_i}h^0\left(X_i, \sO_{X_i}\left(
kr\left(-K_{X_i}+(d-1)E_i\right)-jE_i\right)\right)\\
&=&\sum_{i=1}^m\Bigl(\frac{(kr)^{n+1}}{n!}\int_{\tau_{i-1}}^{\tau_i}
\left(\left(-K_{X_i}+(d-1)E_i-xE_i\right)^{\cdot n}\right)dx\\
&-&\frac{(kr)^n}{2\cdot (n-1)!}\int_{\tau_{i-1}}^{\tau_i}
\left(\left(-K_{X_i}+(d-1)E_i-xE_i\right)^{\cdot n-1}\cdot(K_{X_i}+E_i)\right)dx\Bigr)\\
&+&O(k^{n-1}).
\end{eqnarray*}

Thus we have 
\[
\DF(\sB, \sL)=\frac{r^{2n}((-K_X)^{\cdot n})}{2\cdot(n!)^2}\eta(S), 
\]
where 
\[
\eta(S):=n\sum_{i=1}^m\int_{\tau_{i-1}}^{\tau_i}
(d-x)\left(\left(-K_{X_i}+(d-1-x)E_i\right)^{\cdot n-1}\cdot E_i\right)dx.
\]

Hence we have proved the following: 

\begin{thm}\label{eta_thm}
Under Assumption \ref{fg_assump}, we further assume that the pair 
$(X, -K_X)$ is K-stable $($resp.\ K-semistable$)$. Then we have 
$\eta(S)>0$ $($resp.\ $\eta(S)\geq 0)$. 
\end{thm}

\begin{remark}\label{vol_rmk}
We can show (see \cite[Theorem 5.2]{fjt}) that 
\[
\eta(S)=d\cdot \vol_X(-K_X)-\int_0^{\tau(S)}\vol_{\tilde{X}}(\sigma^*(-K_X)-xE_S)dx, 
\]
where $\vol_X$ is the volume function \cite[Corollary 2.2.45]{L}. 
(We do not use this equality.) 
\end{remark}

\section{Proof of Theorem \ref{mainthm}}\label{proof_section}

Let $X\subset\pr^{n+3}$ 
be a del Pezzo manifold of degree five and of dimension $4\leq n\leq 6$. 
We take a plane $S\subset X$, morphisms $\sigma\colon\tilde{X}\to X$ and 
$\pi\colon\tilde{X}\to \pr^n$ as in Theorem \ref{dP_thm}. 
Since $\tilde{X}$ is a Fano manifold with $\rho(\tilde{X})=2$, $X$ and $S$ satisfies 
Assumption \ref{fg_assump} by Remark \ref{BCHM_rmk}. 
Let $\{(\tau_i, X_i)\}_{1\leq i\leq m}$ be the ample model sequence of $(X, -K_X; I_S)$. 
Then $X_1=\tilde{X}$, $\tau_1=n-1$, $X_2=\pr^n$, $E_2$ is a hyperplane on $X_2$, 
$\tau_2=2n-2$ and $m=2$. 
We note that $E_S=\pr_S(\sN_{S/X}^\vee)$. Let $\xi_E$, $L_E$ be a divisor on $E$ 
corresponds to a tautological line bundle of $\pr_S(\sN_{S/X}^\vee)/S$, the pullback 
of $\sO_{\pr^2}(1)$ on $S$, respectively. 
Then we have 
\[
\sO_{\tilde{X}}(-K_{\tilde{X}}+(n-3-x)E_S)|_{E_S}\simeq\sO_{E_S}(x\xi_E+(n-1)L_E). 
\]
Since 
\begin{eqnarray*}
&&\left(\left(\sO_{E_S}(x\xi_E+(n-1)L_E)\right)^{\cdot n-1}\right)\\
&=&(-c_2(\sN_{S/X})+(n-4)^2)x^{n-1}-(n-1)^2(n-4)x^{n-2}\\
&+&\frac{1}{2}(n-1)^3(n-2)x^{n-3}, 
\end{eqnarray*}
we have 
\begin{eqnarray*}
&&\int_0^{n-1}(n-2-x)\left(\left(\sO_{E_S}(x\xi_E+(n-1)L_E)\right)^{\cdot n-1}\right)dx\\
&=&(n-1)^n\frac{2c_2(\sN_{S/X})+(n-4)(9-n)}{n(n+1)}.
\end{eqnarray*}
On the other hand, 
we have 
\begin{eqnarray*}
\int_{n-1}^{2n-2}(n-2-x)\left(\left(\sO_{\pr^n}(2n-2-x)\right)^{\cdot n-1}
\cdot\sO_{\pr^n}(1)\right)dx
=\frac{-2\cdot(n-1)^n}{n+1}.
\end{eqnarray*}
Therefore, we have 
\[
\eta(S)=\frac{(n-1)^n}{n+1}\left(2c_2(\sN_{S/X})+(n-4)(9-n)-2n\right). 
\]
If $n=4$ or $5$, then $\eta(S)<0$. If $n=6$, then $\eta(S)=0$. 

As a consequence, we have proved Theorem \ref{mainthm}.

\begin{remark}\label{betti_rmk}
From Remark \ref{mainrmk}, 
Theorem \ref{dP_thm} and the theorem of Lefschetz (see also 
\cite[Lemma 10.23]{F}), we can check that 
the fourth Betti number $b_4(X)$ of a del Pezzo manifold $X$ of dimension $n$ and 
degree five is equal to $2$ if $4\leq n\leq 6$. The author does not know any 
Fano manifold $X$ with $(X, -K_X)$ K-unstable and 
$b_{2i}(X)=1$ for any $0\leq i\leq\dim X$. 
\end{remark}

\begin{remark}\label{fjt_rmk}
In the article \cite{fjt}, we treat ``divisorial stability" of Fano manifolds which is a 
necessary condition of K-stability. 
The $X$ in Theorem \ref{mainthm} is divisirially stable by \cite[Corollary 9.3]{fjt}. 
However, Theorem \ref{mainthm} says that the pair $(X, -K_X)$ is not K-stable. 
This implies that K-(semi)stability is strictly stronger than divisorial (semi)stability 
for Fano manifolds.
See also \cite[Remark 9.4]{fjt}.
\end{remark}

\end{document}